\documentclass[12pt]{amsart}
\usepackage{a4wide}
\usepackage{amssymb,bm,mathrsfs}
\usepackage{amsthm}
\usepackage{amsmath}
\usepackage{amscd}
\usepackage{verbatim}
\usepackage[all]{xy}
\usepackage{hyperref}

\usepackage{bookmark}

\usepackage{enumerate}

\numberwithin{equation}{section}

\theoremstyle{plain}
\newtheorem{theorem}{Theorem}[section]
\newtheorem{corollary}[theorem]{Corollary}
\newtheorem{lemma}[theorem]{Lemma}
\newtheorem{proposition}[theorem]{Proposition}

\theoremstyle{definition}
\newtheorem{definition}[theorem]{Definition}

\theoremstyle{remark}
\newtheorem{remark}[theorem]{Remark}

\newcommand{\Q}{\mathbb{Q}}

\newcommand{\N}{\mathbb{N}}
\newcommand{\Z}{\mathbb{Z}}

\renewcommand{\O}{\mathcal{O}}

\renewcommand{\a}{\mathfrak{a}}
\renewcommand{\b}{\mathfrak{b}}
\renewcommand{\c}{\mathfrak{c}}
\renewcommand{\d}{\mathfrak{d}}
\newcommand{\p}{\mathfrak{p}}
\newcommand{\ep}{\varepsilon}

\newcommand{\set}[1]{\left\{ #1 \right\}}

\newcommand{\setm}[1]{\setminus\set{ #1 }}
\newcommand{\br}[1]{\left( #1 \right)}

\newcommand{\se}{\subset}

\newcommand{\defC}{:\;}

\newcommand{\OK}{\O_K}
\newcommand{\OKp}{\O_{K,p}}

\newcommand{\IK}{\mathcal{I}_K}

\DeclareMathOperator{\tr}{tr}

\DeclareMathOperator{\sgn}{sgn}

\newcommand{\und}{\quad\text{and}\quad} 

\newcommand{\kron}[2]{\br{\frac{#1}{#2}}}

\renewcommand{\Re}{\operatorname{Re}}

\begin{document}

\author{Johannes J. Buck}

\address{Fachbereich Mathematik, Technische Universit\"at Darmstadt, Schlossgartenstrasse 7,
D--64289 Darmstadt, Germany}
\email{jbuck@mathematik.tu-darmstadt.de}

\thanks{The author was supported by the DFG Collaborative Research Centre TRR 326 Geometry and Arithmetic of Uniformized Structures, project number 444845124.}

\title[Dirichlet series ass. to rep. numbers of ideals in real quadratic fields]{Dirichlet series associated to representation numbers of ideals in real quadratic number fields}

\date{\today}

\begin{abstract}
In this rather computational paper, we determine certain representation numbers of ideals in real quadratic number fields explicitly in order to obtain a representation of the associated Dirichlet series in terms of Dirichlet L-functions and a generalized divisor sum. A direct and important consequence is that the Dirichlet series has a meromorphic continuation to the whole complex plane and a simple pole at $s=2$ whose residue can be made explicit in terms of the Dirichlet L-functions and the generalized divisor sum.
\end{abstract}

\maketitle

\tableofcontents


\section{Introduction}

Throughout the paper, $K$ is a real quadratic number field of odd discriminant $D$. The discriminant induces the Dirichlet character $\chi_D$ and
its associated L-function $L(s,\chi_D)$ defined by
\[
\chi_D(n) := \kron{D}{n}
\und
L(s,\chi_D): = \sum_{n=1}^\infty \chi_D(n)n^{-s} = \prod_{p} (1-\chi_D(p)p^{-s})^{-1} \text{ for } \Re(s)>1.
\]
With $x \mapsto x'$ we denote the conjugation in $K$, with $N(x):=xx'$ and $\tr(x):=x+x'$ the norm and the trace.
The ring of integers of $K$ is given by $\OK = \Z + \tfrac{1+\sqrt{D}}{2} \Z$.
By $\IK$ we denote the ideal group of $K$ and by $\d \in \IK$ the \emph{different} $\d=(\sqrt{D}) = \sqrt{D}\OK$. Each ideal $\a \in \IK$ has a unique product representation
\[
\a = \prod_{ \substack{\p \se \OK \\ \p \text{ prime}}   } \p^{\nu_p(\a)}
\]
with $\nu_p(\a) \in \Z$.
We call ideals $\a, \b \in \IK$ \emph{coprime} if we have $\nu_\p(\a)\nu_\p(\b) = 0$ for all prime ideals $\p \se \OK$. If $\a$ or $\b$ happen to be elements of $K$ instead of ideals in $\IK$, we tacitly replace them by their principal ideals for the definition of being coprime.

The finite exponential sum $G^b(\a,m,\nu)$ is defined by
\[
G^b(\a,m,\nu) := \sum_{ \substack{ \lambda \in \a\d^{-1} / b\a \\ \frac{N(\lambda)}{N(\a)} \equiv -\frac{m}{D}\: (b\Z) }  } e \br{\tr  \br{\tfrac{\nu\lambda'}{N(\a)b} } }
\]
for $b \in \N$, $\a \in \IK$, $m \in \Z$ and $\nu \in \a\d^{-1}$.
In this paper however, we are only interested in the case $\nu=0$ where $G^b(\a,m,0)$ equals the representation number
\begin{align*}
G^b(\a,m,0)= \# \set{ \lambda \in \a\d^{-1}/b \a \defC \frac{N(\lambda)}{N(\a)} \equiv  -\frac{m}{D}  \pmod b}
\end{align*}
in order to examine the Dirichlet series
\begin{align} \label{Gseries} 
\sum_{b=1}^\infty G^b(\a,m,0)b^{-s}.
\end{align}
We make the representation numbers $G^b(\a,m,0)$ explicit which allows us to prove in Theorem~\ref{Gb-series} the Dirichlet series representation
\begin{align} \label{Gseries-mne0} 
\sum_{b=1}^\infty G^b(\a,m,0)b^{-s}
=|m|^{-s/2}  \frac{\zeta(s-1)}{L(s,\chi_D)} \sigma(\a,m,1-s)
\end{align}
for $m \ne 0$ where $\sigma(\a,m,s)$ is a certain generalized divisor sum only depending on the genus of $\a$ (cf.~Definition~\ref{div-sum-def}). Similarly, for $m=0$ we prove in Corollary~\ref{G-series-m0}
\begin{align} \label{Gseries-m0} 
\sum_{b=1}^\infty G^b(\a,0,0)b^{-s} = \zeta(s-1) \frac{L(s-1,\chi_D)}{L(s,\chi_D)}.
\end{align}
These representations show that the Dirichlet series~\eqref{Gseries} has a meromorphic continuation to the whole complex plane. At $s=2$ it has a simple pole coming from the zeta function $\zeta(s-1)$ which plays a big role in \cite{buckdiss} in the regularization of automorphic Green functions. 
Theorem~\ref{Gb-series} is a generalization of \cite[Lemma 2.10]{bruinier2007borcherds} which states identity~\eqref{Gseries-mne0} in the special case of $\a = \OK$ and $D$ being a prime number.

Along the way of making the representation numbers $G^b(\a,m,0)$ explicit
we give precise formulae for the related representation numbers
\[
N_b(\a,m) := \# R_b(\a,m)
\quad\text{with}\quad
R_b(\a,m) := \set{\lambda \in \a/b\a \defC  \frac{N(\lambda)}{N(\a)} \equiv m \pmod b}
\]
in Proposition~\ref{rep-numbers}. This generalizes \cite[Lemma 3]{zagier1975modular} where precise formulae for those representation numbers are stated in the special case $\a = \OK$. However, a proof for Zagiers Lemma is omitted in \cite{zagier1975modular}, here we provide it.


\section{Reduction to representation numbers \texorpdfstring{$N_b(\a,m)$}{N\_b(a,m)}}

In this small section we accomplish to express the representation numbers $G^b(\a,m,0)$ in terms of the representation numbers $N_b(\a,m)$.
Namely, we have
\begin{align*}
G^b(\a,m,0)
&= \# \set{ \lambda \in \a\d^{-1}/b \a \defC \frac{N(\lambda)}{N(\a)} \equiv  -\frac{m}{D}  \pmod b}\\
&= \# \set{ \lambda \in \a/b\a\d \defC \frac{N(\lambda)}{N(\a)} \equiv m \pmod{bD}}\\
&= D^{-1} \# \set{ \lambda \in \a/bD\a \defC \frac{N(\lambda)}{N(\a)} \equiv m \pmod{bD}} = \frac{N_{bD}(\a,m)}{D}.
\end{align*}
However, the representation numbers $N_b(\a,m)$ are multiplicative in $b$ by the Chinese remainder theorem. Therefore, we obtain for our Dirichlet series the Euler product representation
\begin{align} \label{first-euler} 
\begin{split}
\sum_{b=1}^\infty G^b(\a,m,0)b^{-s}
&= \frac{1}{D} \sum_{b=1}^\infty N_{bD}(\a,m) b^{-s}\\
&= \prod_{p \mid D} \br{ \frac{1}{p} \sum_{r=0}^\infty N_{p^{r+1}}(\a,m) p^{-rs} } \prod_{p \nmid D} \br{ \sum_{r=0}^\infty N_{p^r}(\a,m) p^{-rs} }.
\end{split}
\end{align}
Note that since $D$ is odd each prime divisor $p \mid D$ divides $D$ with multiplicity $1$.

\section{Dependence on the genus only}
In this section we show that the representation numbers $N_b(\a,m)$ depend only on the genus of $\a \in \IK$ and not the ideal $\a$ itself. By our considerations of the last section this also applies to $G^b(\a,m,0)$ and the Dirichlet series~\eqref{Gseries} then.

\begin{lemma} \label{b-square-invariance} 
Let $b \in \N$ and $a \in \Z$ be such that $a$ is a square in $(\Z/b\Z)^\times$. Then for all $\a \in \IK$ and $n \in \Z$ it holds
\[
N_b(\a,n) = N_b(\a,a n).
\]
\end{lemma}
\begin{proof}
The statement follows from the fact that for $c \in \Z$ with $c^2 \equiv a \pmod b$ the map
\[
R_b(\a,n) \to R_b(\a,a n),\quad x \mapsto cx
\]
is bijective:
Let $x \in R_b(\a,n)$. Then we have
\[
\frac{N(cx)}{N(\mathfrak a)} = c^2 \frac{N(x)}{N(\mathfrak a)} \equiv c^2 n \equiv an \pmod b.
\]
The inverse mapping is obtained by inverting $c$ modulo $b$.
\end{proof}

\begin{lemma} \label{one-implies-n} 
Let $b \in \N$ and $\a,\b \in \IK$ with $R_b(\b,1) \ne \emptyset$. Then we have for all $n \in \Z$
\[
N_b(\a,n)=N_b(\a\b,n).
\]
\end{lemma}
\begin{proof}
Because of $R_b(\b,1) \ne \emptyset$ there exists a $\lambda \in \b$ with $N(\lambda)/N(\b) \equiv 1 \pmod b$.
Now for $x \in R_b(\a,n)$ we have $\lambda x \in R_b(\a\b,n)$ because of
\[
\frac{N(\lambda x)}{N(\a \b)} = \frac{N(\lambda)}{N(\b)} \frac{N(x)}{N(\a)} \equiv 1 \cdot n \equiv n \pmod b.
\]
It follows the existence of the map
\[
R(\a,n,b) \to R(\a\b,n,b),\quad x \mapsto \lambda x.
\]
We show that this map is bijective by finding a $\mu \in R_b(\b^{-1},1)$ with $\lambda \mu \equiv 1 \pmod b$ (which induces the inverse mapping then).
Let $p$ be prime with $p \mid b$. Then we have because of $N(\lambda)/N(\b) \equiv 1 \pmod b$ and
\[
\frac{N(\lambda)}{N(\b)} \OK = \frac{\lambda}{\b} \frac{\lambda'}{\b'}
\]
the equality $\nu_\p(\lambda) = \nu_\p(\b)$ for all prime ideals $\p$ over $p$.
That implies that there exists a $\mu \in \b^{-1}$ with $\lambda \mu \equiv 1 \pmod b$.
We obtain $\lambda' \mu' \equiv 1 \pmod b$. Therefore, we get
\[
\frac{N(\mu)}{N(\b^{-1})} \equiv \frac{N(\mu)}{N(\b^{-1})} \frac{N(\lambda)}{N(\b)} = \lambda \mu \lambda' \mu' \equiv 1 \pmod b.
\]
\end{proof}

\begin{lemma} \label{narrow-same-rep-numbers} 
Let $\a,\b \in \IK$ be two equivalent ideals in the narrow sense. Then we have
\[
N_b(\a,n) = N_b(\b,n)
\]
for all $n \in \Z$ and $b \in \N$.
\end{lemma}
\begin{proof}
We write $\b = \lambda \a$ with $\lambda \in K$	and $N(\lambda)>0$. This implies $N((\lambda))=N(\lambda)$. Therefore, the map
\[
R(\a,n,b) \to R(\b,n,b),\quad x \mapsto \lambda x
\]
delivers the aimed bijection.
\end{proof}

\begin{lemma} \label{square-1-nonempty} 
We have $R_b(\a^2,1) \ne \emptyset$ for all $\a \in \IK$ and $b \in \N$.
\end{lemma}
\begin{proof}
We choose a prime ideal $\p$ of the same ideal class as $\a$ coprime to $b$. Then we have that $\a^2$ and $\p^2$ are equivalent in the narrow sense ($\a = \lambda \p$ implies $\a^2 = \lambda^2 \p^2$ with $N(\lambda^2)=N(\lambda)^2>0$). Lemma~\ref{narrow-same-rep-numbers} implies
\[ N_b(\a^2,1) = N_b(\p^2,1). \]
Therefore, it is enough to show $R_b(\p^2,1) \ne \emptyset$.
Now let $m \in \N$ be the order of $\p$ in the ideal class group. Thus, there exists an $x \in \p^2$ with $N(x)>0$ and $\p^{2m}=(x)$. We define
\[
n := \frac{N(x)}{N(\p^2)} = N(\p)^{2(m-1)}
\]
and obtain $x \in R_b(\p^2,n)$. Since $n$ is coprime to $b$ and a square in $\Z$, Lemma~\ref{b-square-invariance} implies $R_b(\p^2,1) \ne \emptyset$.
\end{proof}

\begin{proposition} \label{rep-genus} 
Let $\a,\b \in \IK$ be ideals of the same genus. Then we have
\[
N_b(\a,n) = N_b(\b,n)
\]
for all $n \in \Z$ and $b \in \N$.
\end{proposition}
\begin{proof}
Let $\b = \lambda \a \c^2$ with $N(\lambda)>0$. Then we have
\[
N_b(\a,n) = N_b(\lambda \a,n) = N_b(\lambda \a \c^2,n).
\]
The first equality follows from Lemma~\ref{narrow-same-rep-numbers}, the second with $R_b(\c^2,1) \ne \emptyset$ (Lemma~\ref{square-1-nonempty}) from Lemma~\ref{one-implies-n}.
\end{proof}

\section{Evaluation of Gauss sums}

We define the Gauss sum
\[
G_b(\a,a) = \sum_{\lambda \in \a/b \a} e \br{ \frac{aN(\lambda)}{bN(\a)} }
\]
for $\a \in \IK$, $a \in \Z$ and obtain for prime powers $b=p^\beta$
\begin{align} \label{Nam-by-gauss} 
\begin{split}
&\frac{1}{p^\beta} \sum_{a \in \Z / p^\beta \Z} G_{p^\beta}(\a,a) e \br{-\frac{am}{p^\beta}}\\
= &\frac{1}{p^\beta} \sum_{a \in \Z / p^\beta \Z} \sum_{\lambda \in \a/p^\beta \a} e \br{ \frac{aN(\lambda)/N(\a)}{p^\beta} } e \br{-\frac{am}{p^\beta}}\\
= &\sum_{\lambda \in \a/p^\beta \a} \frac{1}{p^\beta} \sum_{a \in \Z / p^\beta \Z}  e \br{ \frac{a(N(\lambda)/N(\a)-m)}{p^\beta} }\\
= &N_{p^\beta}(\a,m)
\end{split}
\end{align}
because of 
\[
\sum_{a \in \Z / p^\beta \Z}  e \br{ \frac{a c}{p^\beta} }
= \begin{cases}
p^\beta,&\quad c \equiv 0 \pmod{p^\beta},\\
0,&\quad c \not\equiv 0 \pmod{p^\beta}.
\end{cases}
\]
Hence, in order to evaluate the multiplicative representation numbers $N_b(\a,m)$, it is enough to evaluate the Gauss sums $G_{p^\beta}(\a,a)$. We do this in this section. The idea to use Gauss sums to evaluate representation numbers was already present in Siegel's work. In this particular instance we have strong similarities with \cite{bruinier2021arithmetic}.

We start by computing the Gauss sums for $\a = \OK$. In order to do so we recall the evaluation of the classical Gauss sum
\begin{align} \label{classic-gs} 
\sum_{x \in \Z / c\Z} e \kron{ax^2}{c} = \ep_c \sqrt{c} \kron{a}{c}
\end{align}
for odd $c>0$ with coprime $a \in \Z$.
Here, $\ep_c=1$ for $c \equiv 1 \pmod 4$ und $\ep_c=i$ otherwise.

\begin{lemma} \label{ok-gs} 
Let $p$ be prime, $b=p^\beta \in \N$ and $a = a_0 p^\alpha \in \Z$ with $p^\alpha=\gcd(a,b)$. Then it holds
\[
G_b(\OK,a) =
\begin{cases}
p^{2\beta},&\quad \alpha= \beta,\\
(\chi_D(p)p)^{\alpha+\beta},&\quad p \nmid D,\\
\ep_p p^{\alpha+\beta+1/2} \kron{a_0}{p} \kron{-D_0}{p}^{\alpha+\beta+1} ,&\quad p\mid D, \alpha < \beta.
\end{cases}
\]
For the last case we define $D_0=D/p$.
\end{lemma}
\begin{proof}
The case $p=2$ is discussed in the proof of \cite[Lemma 2.6.1]{bruinier2021arithmetic} in detail.
We continue with $p$ being odd. By definition of $G_b(\OK,a)$ we have
\begin{align*}
G_b(\OK,a)
&= \sum_{k,j \ (b)} e \br{ \frac{ a N \br{ k +j \tfrac{1+\sqrt{D}}{2}} }{b} }.
\end{align*}
It is easy to check that
\[
N \br{ k +j \tfrac{1+\sqrt{D}}{2}} \equiv N \br{ k +j l (1+\sqrt{D})} \pmod b
\]
with $l \in \Z$ satisfying $2 l \equiv 1 \pmod b$. Thus, we obtain
\begin{align*}
G_b(\OK,a)
&= \sum_{k,j \ (b)} e \br{ \frac{ a(k+jl)^2 - a (jl)^2D }{b} }\\
&= \sum_{j} e \br{ \frac{  - a (jl)^2D }{b} } \sum_{k} e \br{ \frac{ a(k+jl)^2  }{b} } \\
&= \sum_{j} e \br{ \frac{  - a j^2 D }{b} } \sum_{k} e \br{ \frac{ ak^2  }{b} }.
\end{align*}
In case $\alpha=\beta$ we sum only ones and the statement is clear.
Otherwise we use the formula~\eqref{classic-gs} for the classical Gauss sum
\[
\sum_{k\ (b)} e \br{ \frac{ ak^2  }{b} } 
= p^\alpha \sum_{k\ (p^{\beta-\alpha})} e \br{ \frac{ a_0k^2  }{p^{\beta-\alpha}} }
= p^\alpha \ep_{p^{\beta-\alpha}} \sqrt{p^{\beta-\alpha}} \kron{a_0}{p^{\beta-\alpha}}.
\]
Analogously, we obtain in case $p \nmid D$
\[
\sum_{j} e \br{ \frac{  - a j^2 D }{b} }  = p^\alpha \ep_{p^{\beta-\alpha}} \sqrt{p^{\beta-\alpha}} \kron{-a_0D}{p^{\beta-\alpha}}.
\]
As product we obtain
\[
p^{\alpha+\beta} \ep_{p^{\beta-\alpha}}^2 \kron{-D}{p^{\beta-\alpha}}.
\]
A case distinction by $p \pmod 4$ delivers the aimed result.
In case $p \mid D$ we proceed analogously and draw $p^{\alpha+1}$ out:
\begin{align*}
\sum_{j \ (b)} e \br{ \frac{  - a D j^2 }{b} }
&= p^{\alpha+1} \sum_{j \ (p^{\beta-\alpha-1})} e \br{ \frac{  - a_0 D_0j^2 }{p^{\beta-\alpha-1}} }\\
&= p^{\alpha+1} \ep_{p^{\beta-\alpha-1}} p^{(\beta-\alpha-1)/2} \kron{- a_0 D_0}{p^{\beta-\alpha-1}}.
\end{align*}
We form the product
\begin{align*}
p^\alpha \ep_{p^{\beta-\alpha}} &\sqrt{p^{\beta-\alpha}} \kron{a_0}{p^{\beta-\alpha}} p^{\alpha+1} \ep_{p^{\beta-\alpha-1}} p^{(\beta-\alpha-1)/2} \kron{- a_0 D_0}{p^{\beta-\alpha-1}}\\
&= \ep_p p^{\alpha+\beta+1/2} \kron{a_0}{p}^{\beta-\alpha} \kron{- a_0 D_0}{p}^{\beta-\alpha-1}\\
&= \ep_p p^{\alpha+\beta+1/2} \kron{a_0}{p} \kron{-D_0}{p}^{\alpha+\beta+1}.
\end{align*}
\end{proof}

\begin{remark} \label{coprime-legendre} 
Let $p$ be an odd prime and $\a \in \IK$ coprime to $p$. Then $N(\a)$ might not be integral, but we have $\nu_p(N(\a))=0$. Hence, we can interpret the rational number $N(\a)$ as element of $(\Z/p\Z)^\times$ and the Legendre symbol
\[
\kron{N(\a)}{p}
\]
is well defined and non-zero. That is how we read such expressions in this paper.
\end{remark}

\begin{lemma} \label{a-gs} 
Let $p$ be prime, $b=p^\beta \in \N$, $\a \in \IK$ and $a = a_0 p^\alpha \in \Z$ with $p^\alpha=\gcd(a,b)$. Then it holds
\[
G_b(\a,a) =
\begin{cases}
p^{2\beta},&\quad \alpha= \beta,\\
(\chi_D(p)p)^{\alpha+\beta},&\quad p \nmid D,\\
\ep_p p^{\alpha+\beta+1/2} \kron{a_0 N(\a)}{p} \kron{-D_0}{p}^{\alpha+\beta+1} ,&\quad p\mid D, \alpha < \beta.
\end{cases}
\]
In the last case we additionally assume that $\a$ is coprime to $p$ (cf.~Remark~\ref{coprime-legendre}).
\end{lemma}
\begin{proof}
The Gauss sum $G_{p^\beta}(\a,a)$ depends only on $\a_p := \a \otimes_{\Z} \Z_p$. However, $\a_p$ is generated as $\OKp := \OK \otimes_{\Z} \Z_p$ module by a single element $x \in \a$. Namely, one easily verifies
\[
\a_p \cap K = \Z_{(p)} \a
\quad\text{with}\quad
\Z_{(p)} := \Z_p \cap \Q = \set{\frac{a}{b} \defC a,b \in \Z, p \nmid b}.
\]
This implies $\OKp x = \a_p$ for all $x \in \a$ which satisfy $\nu_\p(x)=\nu_\p(\a)$ for the prime ideals $\p$ over $p$.
Thus, the integer $N(x)/N(\a)$ is coprime to $p$ for such $x \in \a$.
We obtain for such a generator $x \in \a$
\[
G_{p^\beta}(\a,a) = G_{p^\beta} \br{\OK,a \frac{N(x)}{N(\a)} }.
\]
Using Lemma~\ref{ok-gs}  and the fact that $N(x)/N(\a)$ is coprime to $p$ the only thing left to show is
\[
\kron{a_0 N(x)/N(\a)}{p} = \kron{a_0 N(\a)}{p}
\]
in case $p \mid D$. This follows from 
\[
\kron{N(x)}{p} = 1
\]
where we intepret the rational number $N(x) \in \Z_{(p)}$ as element of $(\Z / p \Z)^\times$. The equation follows from
\[
N(x) = c^2-d^2D \equiv c^2 \pmod p
\]
for $x= c + d\sqrt{D}$.
\end{proof}

\section{Evaluation of representation numbers \texorpdfstring{$N_b(\a,m)$}{N\_b(a,m)}}

Lemma \ref{a-gs} together with equation~\eqref{Nam-by-gauss} allows us to compute the representation numbers which we do in the upcoming proposition.
\begin{proposition} \label{rep-numbers} 
Let $b=p^\beta$ and $m = m_0 p^\nu$ with $p^\nu=\gcd(m,b)$.
\begin{enumerate}[(i)]
\item In case $\chi_D(p)=1$
\[
N_{b}(\a,m) =
\begin{cases}
(\nu+1) (p-1) p^{\beta-1} ,&\quad \nu < \beta,\\
(\beta+1) p^\beta -\beta p^{\beta-1},&\quad \nu = \beta.
\end{cases}
\]
\item In case $\chi_D(p)=-1$
\[
N_{b}(\a,m) =
\begin{cases}
(p+1) p^{\beta-1} ,&\quad \nu < \beta,\ 2 \mid \nu,\\
0,&\quad \nu < \beta,\ 2 \nmid \nu,\\
p^\beta ,&\quad \nu = \beta,\ 2 \mid \beta,\\
p^{\beta-1} ,&\quad \nu = \beta,\ 2 \nmid \beta.
\end{cases}
\]
\item In case $\chi_D(p)=0$ with $\a \in \IK$ coprime to $p$ and $D_0:=D/p$ 
\[
N_{b}(\a,m) = 
\begin{cases}
\br{1 + \kron{-D_0}{p}^\nu \kron{m_0 N(\a)}{p}} p^\beta,&\quad \nu < \beta,\\
p^\beta,&\quad \nu = \beta.
\end{cases}
\]
\end{enumerate}
\end{proposition}
\begin{proof}
Let $p \nmid D$. By equation~\eqref{Nam-by-gauss} and Lemma \ref{a-gs} we have
\begin{align*}
N_{p^\beta}(\a,m)
&= \frac{1}{p^\beta} \sum_{a \in \Z / p^\beta \Z} G_{p^\beta}(\a,a) e \br{-\frac{am}{p^\beta}}\\
&= \frac{1}{p^\beta} \sum_{\alpha=0}^\beta \sum_{a_0 \in ( \Z / p^{\beta-\alpha}\Z)^\times} (\chi_D(p)p)^{\alpha+\beta} e \br{-\frac{a_0 p^\alpha m}{p^\beta}}\\
&= \chi_D(p)^\beta \sum_{\alpha=0}^\beta (\chi_D(p)p)^{\alpha} \sum_{a_0 \in ( \Z / p^{\beta-\alpha}\Z)^\times}  e \br{-\frac{a_0 m}{p^{\beta-\alpha}}}.
\end{align*}
For $\alpha \ge \beta - \nu$ we obtain
\begin{align*}
\sum_{a_0 \in ( \Z / p^{\beta-\alpha}\Z)^\times}  e \br{-\frac{a_0 m}{p^{\beta-\alpha}}}
&= \sum_{a_0 \in ( \Z / p^{\beta-\alpha}\Z)^\times}  e \br{- a_0 m_0p^{\alpha-\beta+\nu}}\\
&= \sum_{a_0 \in ( \Z / p^{\beta-\alpha}\Z)^\times} 1\\
&= \begin{cases}
(p-1)p^{\beta-\alpha-1} , &\quad \alpha<\beta,\\
1, &\quad \alpha=\beta.
\end{cases}
\end{align*}
For $\alpha < \beta - \nu$ we obtain instead
\begin{align*}
\sum_{a_0 \in ( \Z / p^{\beta-\alpha}\Z)^\times}  e \br{-\frac{a_0 m}{p^{\beta-\alpha}}}
&= \sum_{a_0 \in ( \Z / p^{\beta-\alpha}\Z)^\times}  e \br{-\frac{a_0 m_0}{p^{\beta-\alpha-\nu}}}\\
&= p^\nu \mu(p^{\beta-\alpha-\nu}) =
\begin{cases}
0, &\quad \alpha < \beta-\nu-1,\\
-p^\nu, &\quad \alpha = \beta-\nu-1.
\end{cases}
\end{align*}
Thus, we obtain in case $\nu < \beta$
\begin{align*}
N_{p^\beta}(\a,m)
&= \chi_D(p)^\beta \sum_{\alpha=0}^\beta (\chi_D(p)p)^{\alpha} \sum_{a_0 \in ( \Z / p^{\beta-\alpha}\Z)^\times}  e \br{-\frac{a_0 m}{p^{\beta-\alpha}}}\\
&= \chi_D(p)^\beta \br{ (\chi_D(p)p)^{\beta-\nu-1} (-p^\nu) +  \sum_{\alpha=\beta-\nu}^{\beta-1} (\chi_D(p)p)^{\alpha} (p-1)p^{\beta-\alpha-1} + (\chi_D(p)p)^{\beta} }\\
&= \chi_D(p)^\beta \br{ -\chi_D(p)^{\beta-\nu-1} p^{\beta-1} +  \sum_{\alpha=\beta-\nu}^{\beta-1} \chi_D(p)^\alpha (p-1)p^{\beta-1} + \chi_D(p)^\beta p^{\beta} }\\
&=-\chi_D(p)^{\nu+1} p^{\beta-1} +   (p-1)p^{\beta-1} \sum_{\alpha=\beta-\nu}^{\beta-1} \chi_D(p)^{\alpha+\beta} +  p^{\beta}.
\end{align*}
With $\chi_D(p)=1$ this results to
\[
N_{p^\beta}(\a,m) = - p^{\beta-1} +   (p-1)p^{\beta-1} \sum_{\alpha=\beta-\nu}^{\beta-1} 1 +  p^{\beta} = (\nu+1)(p-1)p^{\beta-1}.
\]
With $\chi_D(p)=-1$ the middle sum collapses for $\nu$ even and we obtain
\[
N_{p^\beta}(\a,m) = - (-1)^{\nu+1} p^{\beta-1} +  p^{\beta} = (p+1)p^{\beta-1}.
\]
For $\nu$ odd we obtain
\[
N_{p^\beta}(\a,m) = - (-1)^{\nu+1} p^{\beta-1} +   (p-1)p^{\beta-1} (-1)^{2\beta-\nu} +  p^{\beta}
=- p^{\beta-1}   -(p-1)p^{\beta-1}  +  p^{\beta} = 0.
\]
To finish the proof of the major case $p \nmid D$ we are left with the subcase $\nu = \beta$:
\begin{align*}
N_{p^\beta}(\a,m)
&= \chi_D(p)^\beta \sum_{\alpha=0}^\beta (\chi_D(p)p)^{\alpha} \sum_{a_0 \in ( \Z / p^{\beta-\alpha}\Z)^\times}  e \br{-\frac{a_0 m}{p^{\beta-\alpha}}}\\
&= \chi_D(p)^\beta  \br{ (\chi_D(p)p)^{\beta}  + \sum_{\alpha=0}^{\beta-1} (\chi_D(p)p)^{\alpha} (p-1)p^{\beta-\alpha-1}  }\\
&= \chi_D(p)^\beta  \br{ (\chi_D(p)p)^{\beta}  + (p-1) p^{\beta-1}\sum_{\alpha=0}^{\beta-1} \chi_D(p)^{\alpha}   }.
\end{align*}
Therefore, in case $\chi_D(p)=1$ we obtain
\[
N_{p^\beta}(\a,m) = p^\beta + (p-1) p^{\beta-1} \beta= (\beta+1)p^\beta - \beta p^{\beta-1}.
\]
In case $\chi_D(p)=-1$ we obtain for even $\beta$
\[
N_{p^\beta}(\a,m) = p^{\beta}  + (p-1) p^{\beta-1} \cdot 0 = p^{\beta} 
\]
and for odd $\beta$
\[
N_{p^\beta}(\a,m) = - \br { -p^{\beta}  + (p-1) p^{\beta-1} \cdot (+1)} = p^{\beta-1}.
\]
This finishes the proof of the major case $p \nmid D$ (case (i) and (ii)).
Now let $p \mid D$ which is case (iii).
We use equation~\eqref{Nam-by-gauss} and Lemma \ref{a-gs} again to obtain
\begin{align*}
N_{p^\beta}(\a,m)
&= \frac{1}{p^\beta} \sum_{a \in \Z / p^\beta \Z} G_{p^\beta}(\a,a) e \br{-\frac{am}{p^\beta}}\\
&= p^\beta + \frac{1}{p^\beta} \sum_{\alpha=0}^{\beta-1} \sum_{a_0 \in ( \Z / p^{\beta-\alpha}\Z)^\times} \ep_p p^{\alpha+\beta+1/2} \kron{a_0 N(\a)}{p} \kron{-D_0}{p}^{\alpha+\beta+1}  e \br{-\frac{a_0 p^\alpha m}{p^\beta}}\\
&= p^\beta + \kron{N(\a)}{p} \sum_{\alpha=0}^{\beta-1}  \frac{\ep_p p^{\alpha+\beta+1/2}}{p^\beta} \kron{-D_0}{p}^{\alpha+\beta+1}  \sum_{a_0 \in ( \Z / p^{\beta-\alpha}\Z)^\times}  \kron{a_0}{p}  e \br{-\frac{a_0 p^\alpha m}{p^\beta}}\\
&= p^\beta +\kron{N(\a)}{p} \sum_{\alpha=0}^{\beta-1}  \frac{\ep_p p^{\alpha+\beta+1/2}}{p^\beta} \kron{-D_0}{p}^{\alpha+\beta+1}  \sum_{a \in \Z / p^{\beta-\alpha}\Z}  \kron{a}{p}  e \br{\frac{-m_0a}{p^{\beta-\alpha-\nu}}}.
\end{align*}
For $r>1$ and $c$ coprime to the prime $p$ one has
\begin{align} \label{further-gs} 
\sum_{a \in \Z / p^r \Z}  \kron{a}{p}  e \br{\frac{ac}{p^r}} = 0.
\end{align}
We obtain
\[
\sum_{a \in \Z / p^{\beta-\alpha}\Z}  \kron{a}{p}  e \br{\frac{-m_0a}{p^{\beta-\alpha-\nu}}} =
\begin{cases}
0 , &\quad \beta-\alpha-\nu>1,\\
p^\nu \ep_p  \kron{-m_0}{p}\sqrt{p} , &\quad \beta-\alpha-\nu=1,\\
0 , &\quad \beta-\alpha-\nu \le 0.
\end{cases}
\]
The zero in the first line is due to equation~\eqref{further-gs}. The reason for the zero in the last line is that $(\Z / p^{\beta-\alpha}\Z)^\times$ has equally many squares and non-squares.

In case $\nu=\beta$ we are for all $\alpha$ in the third line and we obtain $N_{p^\beta}(\a,m)= p^\beta$.
In case $\nu \le \beta-1$ there is precisely one summand non-zero, the one of index $\alpha=\beta-\nu-1$ and we obtain
\begin{align*}
N_{p^\beta}(\a,m)
&= p^\beta +\kron{N(\a)}{p} \sum_{\alpha=0}^{\beta-1} \ep_p p^{\alpha+1/2}\kron{-D_0}{p}^{\alpha+\beta+1}  \sum_{a \in \Z / p^{\beta-\alpha}\Z}  \kron{a}{p}  e \br{\frac{-m_0a}{p^{\beta-\alpha-\nu}}}\\
&= p^\beta + \kron{N(\a)}{p} \ep_p p^{\beta-\nu-1+1/2}\kron{-D_0}{p}^{\beta-\nu-1+\beta+1} p^\nu \ep_p  \kron{-m_0}{p}\sqrt{p}\\
&= p^\beta + \kron{N(\a)}{p} \ep_p^2 p^{\beta}\kron{-D_0}{p}^\nu \kron{-m_0}{p}\\
&= \br{1 + \kron{-D_0}{p}^\nu \kron{m_0 N(\a)}{p}} p^\beta.
\end{align*}
This finishes the proof.
\end{proof}
\begin{remark}
The restriction $\a \in \IK$ coprime to $p$ in case (iii) of Proposition~\ref{rep-numbers} is no actual restriction since by Proposition~\ref{rep-genus} the representation numbers only depend on the genus. Hence, one only needs to replace $\a$ by a $\b \in \IK$ coprime to $p$ of the same genus.
\end{remark}

\section{Evaluation of Euler factors}
Now, having precise formulae for the representation numbers $N_b(\a,m)$ at hand,
we can use them to evaluate the series from equation~\eqref{first-euler} in the next proposition.

\begin{proposition} \label{rep-prime-series}
Let $p$ be prime, $\Re(s)>1$, $q := p^{1-s}$, $m = m_0 p^\nu \in \Z$ with $p \nmid m_0$ and $\a \in \IK$ coprime to $D$.
Then we have
\begin{align*}
\sum_{r=0}^\infty N_{p^r}(\a,m) p^{-rs} &= \frac{p-\chi_D(p)q}{p(1-q)} \cdot \frac{1- \chi_D(p)^{\nu+1} q^{\nu+1}}{1- \chi_D(p) q},& p \nmid D,\\
\frac{1}{p} \sum_{r=0}^\infty N_{p^{r+1}}(\a,m) p^{-rs} &= \frac{1 + \kron{-D/p}{p}^\nu \kron{N(\a) m_0}{p} q^\nu }{1-q}, &  p \mid D.
\end{align*}
\end{proposition}
\begin{proof}
We use Proposition~\ref{rep-numbers} and start in the case $\chi_D(p)=1$:
\begin{align*}
&\sum_{r=0}^\infty N_{p^r}(\a,m) p^{-rs}\\
&= \sum_{r=0}^{\nu} \br{(r+1) p^r -r p^{r-1}}  p^{-rs} + \sum_{r=\nu+1}^\infty (\nu+1) p^{r-1} (p-1)p^{-rs}\\
&=\frac{p-1}{p} \sum_{r=0}^{\nu} r q^r + \sum_{r=0}^{\nu} q^r +  (\nu+1) \frac{p-1}{p} \sum_{r=\nu+1}^\infty  q^r\\
&=\frac{p-1}{p} \frac{q(\nu q^{\nu+1}-(\nu+1)q^\nu+1)}{(q-1)^2} + \frac{1-q^{\nu+1}}{1-q}  +  (\nu+1) \frac{p-1}{p}  \frac{q^{\nu+1}}{1-q} \\
&= \frac{p-1}{p (q-1)^2} \br{ q(\nu q^{\nu+1}-(\nu+1)q^\nu+1) + (\nu+1)q^{\nu+1}(1-q)  }
+ \frac{1-q^{\nu+1}}{1-q} \\
&= \frac{p-1}{p (q-1)^2} \br{ \nu q^{\nu+2}- \nu q^{\nu+1} -q^{\nu+1} + q + \nu q^{\nu+1} - \nu q^{\nu+2} +q^{\nu+1} - q^{\nu+2} }
+ \frac{1-q^{\nu+1}}{1-q} \\
&= \frac{p-1}{p (q-1)^2} \br{   q  - q^{\nu+2} } + \frac{1-q^{\nu+1}}{1-q} 
= \frac{pq  - pq^{\nu+2}-q  + q^{\nu+2} + p-pq^{\nu+1}-pq+ pq^{\nu+2} }{p (q-1)^2} \\
&= \frac{  -q  + q^{\nu+2} + p-pq^{\nu+1}  }{p (q-1)^2}
= \frac{ (p-q)(1-q^{\nu+1})}{p (q-1)^2}.
\end{align*}
We continue with $\chi_D(p)=-1$ for even $\nu$:
\begin{align*}
\sum_{r=0}^\infty N_{p^r}(\a,m) p^{-rs}
&= \sum_{r=0}^{\nu/2} p^{2r} p^{-2rs} + \sum_{r=0}^{\nu/2-1} p^{2r} p^{-(2r+1)s} + \sum_{r=\nu+1}^\infty (p+1) p^{(r-1)(1-s)}\\
&= \sum_{r=0}^{\nu/2} q^{2r} + \frac{q}{p} \sum_{r=0}^{\nu/2-1} q^{2r} + \frac{p+1}{p} \sum_{r=\nu+1}^\infty q^r\\
&= \frac{1-q^{\nu+2}}{1-q^2} + \frac{q}{p} \frac{1-q^{\nu}}{1-q^2} + \frac{p+1}{p} \frac{q^{\nu+1}}{1-q}\\
&= \frac{p-pq^{\nu+2} +q-q^{\nu+1} + (p+1)q^{\nu+1}(q+1) }{p(1-q^2)}\\
&= \frac{p-pq^{\nu+2} +q-q^{\nu+1} +  pq^{\nu+2} + p q^{\nu+1} + q^{\nu+2} + q^{\nu+1}}{p(1-q^2)}\\
&= \frac{p +q + p q^{\nu+1} + q^{\nu+2} }{p(1-q^2)}
=  \frac{(p+q)(1+q^{\nu+1})}{p(1-q^2)}.
\end{align*}
Now, let $\nu$ be odd:
\begin{align*}
\sum_{r=0}^\infty N_{p^r}(\a,m) p^{-rs}
= \sum_{r=0}^{(\nu-1)/2} p^{2r} p^{-2rs} + p^{2r} p^{-(2r+1)s} 
= \frac{p+q}{p} \sum_{r=0}^{(\nu-1)/2} q^{2r} 
= \frac{(p+q)(1-q^{\nu+1})}{p(1-q^2)}.
\end{align*}
Lastly, we consider the case $p \mid D$ and define
\[
\sigma := \kron{-D/p}{p}^\nu \kron{N(\a) m_0}{p}.
\]
We compute
\begin{align*}
\frac{1}{p} \sum_{r=0}^\infty N_{p^{r+1}}(\a,m) p^{-rs}
&= \frac{1}{p}\sum_{r=0}^{\nu-1} p^{r+1} p^{-rs} + \frac{1}{p}\sum_{r=\nu}^{\infty} \br{1+ \sigma } p^{r+1} p^{-rs}\\
&= \sum_{r=0}^{\nu-1} q^r + \sum_{r=\nu}^{\infty} \br{1+  \sigma } q^r
= \frac{1-q^\nu  +  \br{1+  \sigma } q^\nu }{1-q} 
= \frac{1 +  \sigma q^\nu }{1-q}.
\end{align*}
\end{proof}

\begin{corollary} \label{G-series-m0} 
For $m=0$ and $\Re(s)>2$ we obtain
\[
\sum_{b=1}^\infty G^b(\a,0,0)b^{-s} = \zeta(s-1) \frac{L(s-1,\chi_D)}{L(s,\chi_D)}.
\]
\end{corollary}
\begin{proof}
By equation~\eqref{first-euler} and Proposition~\ref{rep-prime-series} (set $\nu = \infty$) we obtain
\begin{align*}
\sum_{b=1}^\infty G^b(\a,0,0)b^{-s}
&= \prod_{p \mid D} \br{ \frac{1}{p} \sum_{r=0}^\infty N_{p^{r+1}}(\a,0) p^{-rs} } \prod_{p \nmid D} \br{ \sum_{r=0}^\infty N_{p^r}(\a,0) p^{-rs} }\\
&= \prod_{p \mid D} \frac{1  }{1-q} \cdot \prod_{p \nmid D} \br{\frac{p-\chi_D(p)q}{p(1-q)}  \frac{1}{1- \chi_D(p) q} }\\
&= \prod_{p} \frac{1}{1-p^{1-s}} \cdot  \frac{1-\chi_D(p)p^{-s}}{1- \chi_D(p) p^{1-s}}\\
&= \zeta(s-1) \frac{L(s-1,\chi_D)}{L(s,\chi_D)}.
\end{align*}
\end{proof}

\section{The generalized divisor sum}

In this section we define the generalized divisor sum and compute its Euler product.

\begin{definition} \label{div-sum-def} 
Let $m \in \Z \setm{0}$ and $\a \se \OK$ be coprime to $D$. We define
\[
\sigma(\a,m,s) = |m|^{(1-s)/2} \sum_{d \mid m} d^s \prod_{p \mid D} (\chi_{D(p)}(d) + \chi_{D(p)}(N(\a)m/d)).
\]
The product ranges over all prime divisors $p$ of $D$ and $D(p) \in \set{\pm p}$ such that $D(p)$ is a discriminant, i.e. $D(p) \equiv 1 \pmod 4$.
Now for arbitrary $\b \in \IK$ we define
$\sigma(\b,m,s) := \sigma(\a,m,s)$
with $\a \se \OK$ coprime to $D$ taken from the same genus as $\b$.
\end{definition}
\begin{remark}
The divisor sum is actually well defined and depends only on the genus of the input ideal because for $\a,\b \se \OK$ of the same genus both coprime to $D$ we have $\chi_{D(p)}(N(\a))=\chi_{D(p)}(N(\b))$ for all primes $p \mid D$.

Since the divisor sum $\sigma(\a,m,s)$ depends only on the genus of $\a$, we can assume for the rest of this section $\a \se \OK$ and that $\a$ is coprime to $D$.
\end{remark}

\begin{lemma} \label{sigma-functional} 
The divisor sum $\sigma(\a,m,s)$ satisfies the functional equation
\[
\sigma(\a,m,s) = \sigma(\a,m,-s).
\]
\end{lemma}
\begin{proof}
We compute
\begin{align*}
\sigma(\a,m,-s)
&= |m|^{(1+s)/2} \sum_{d \mid m} d^{-s} \prod_{p \mid D} (\chi_{D(p)}(d) + \chi_{D(p)}(N(\a)m/d))\\
&= |m|^{(1+s)/2} \sum_{d \mid m} \br{\frac{|m|}{d}}^{-s} \prod_{p \mid D} (\chi_{D(p)}(|m|/d) + \chi_{D(p)}(N(\a)d \sgn(m)))\\
&= |m|^{(1-s)/2} \sum_{d \mid m} d^{s}  \prod_{p \mid D}  \chi_{D(p)}(N(\a)\sgn(m)) (\chi_{D(p)}( N(\a)m/d) + \chi_{D(p)}(d))\\
&= \sigma(\a,m,s) \prod_{p \mid D}  \chi_{D(p)}(N(\a)\sgn(m)).
\end{align*}
The claim follows with
\[
 \prod_{p \mid D} \chi_{D(p)}(N(\a)\sgn(m))= \chi_D(N(\a)) \chi_D(\sgn(m)) = 1.
\]
Note that $\chi_D(\sgn(m))=1$ is true because $D$ is positive (even though $D(p)$ might be negative for some $p \mid D$).
\end{proof}

\begin{lemma}
Let $m \in \Z \setm{0}$ and $\a \se \OK$ be coprime to $D$.
Then $s \mapsto \sigma(\a,m,s)$ is the zero function if and only if $\chi_{D(p)}(N(\a)m)=-1$ for a prime $p \mid D$.
\end{lemma}
\begin{proof}
Since the set of functions $\set{ d^s : d \in \N }$ is linearly independent, the function $s \mapsto \sigma(\a,m,s)$ is the zero function if and only if
\[
\prod_{p \mid D} (\chi_{D(p)}(d) + \chi_{D(p)}(N(\a)m/d)) = 0
\]
for all $d \mid m$. For $d=1$ the product has the form
\[
\prod_{p \mid D} (1 + \chi_{D(p)}(N(\a)m)).
\]
This product is zero if and only if $\chi_{D(p)}(N(\a)m)=-1$ for a prime $p \mid D$. This proves the first direction.

Now let $p \mid D$ with $\chi_{D(p)}(N(\a)m)=-1$. This implies that $p \nmid m$. Therefore, $p$ is coprime to all $d \mid m$ and we have
\[
\chi_{D(p)}(d) + \chi_{D(p)}(N(\a)m/d) = \chi_{D(p)}(d) (1 + \chi_{D(p)}(N(\a)m)) = \chi_{D(p)}(d) \cdot 0 =0
\]
which proves the other direction.
\end{proof}

\begin{lemma} \label{almost-euler-sigma} 
For $m \ne 0$ we have the almost Euler product expansion
\begin{align*}
\sum_{d \mid m} d^s &\prod_{p \mid D} (\chi_{D(p)}(d) + \chi_{D(p)}(N(\a)m/d))\\
&= \sum_{D_1D_2=D} \chi_{D_1}(m_{D_2}) \chi_{D_2}(N(\a) m_0 m_{D_1}) m_{D_2}^{s} \prod_{p \nmid D} \frac{1- (\chi_D(p) p^s )^{\nu_p(m)+1} }{ 1- \chi_D(p) p^s}.
\end{align*}
Here, the sum in the first line sums only over the positive divisors of $m$ (even if $m$ is negative) and the sum in the second line sums over all decompositions $D_1D_2=D$ into two discriminants. For example for $D=21$ we have the four decompositions $21= 1 \cdot 21 = (-3) \cdot (-7) = (-7) \cdot (-3) = 21 \cdot 1$. A given decomposition $D_1D_2=D$ induces a decomposition $m=m_0 m_{D_1} m_{D_2}$ where
\[
m_{D_i} := \prod_{p \mid D_i} p^{\nu_p(m)}.
\]
\end{lemma}
\begin{proof}
By expansion of the product we obtain
\begin{align*}
\sum_{d \mid m} d^s \prod_{p \mid D} (\chi_{D(p)}(d) + \chi_{D(p)}(N(\a)m/d))
=  \sum_{D_1D_2=D} \chi_{D_2}(N(\a)) \sum_{d \mid m} d^s \chi_{D_1}(d) \chi_{D_2}(m/d).
\end{align*}
The advantage of this representation is that the latter sum
\[
\sum_{d \mid m} d^s \chi_{D_1}(d) \chi_{D_2}(m/d)
\]
is multiplicative in $m$. Therefore, we only have to evaluate it for each prime power dividing $m$ and $m=-1$ in case the original $m$ is negative.
We obtain
\[
\sum_{d \mid p^\nu} d^s \chi_{D_1}(d)\chi_{D_2}(p^\nu/d) = 
\begin{cases}
\chi_{D_2}(p)^\nu \frac{1- (\chi_D(p) p^s )^{\nu+1} }{ 1- \chi_D(p) p^s},&\quad p \nmid D,\\
\chi_{D_2}(p)^\nu,&\quad p \mid D_1,\\
\chi_{D_1}(p)^\nu p^{\nu s},&\quad p \mid D_2,
\end{cases}
\]
and for $m=-1$
\[
\sum_{d \mid m} d^s \chi_{D_1}(d) \chi_{D_2}(m/d) = \chi_{D_2}(-1) = \sgn(D_2).
\]
We conclude
\begin{align*}
\sum_{d \mid m} d^s \prod_{p \mid D} (\chi_{D(p)}(d) + \chi_{D(p)}(N(\a)m/d))\\
= \sum_{D_1D_2=D} \chi_{D_2}(\sgn(m)N(\a)) \prod_{p \mid D_1} \chi_{D_2}(p)^{\nu_p(m)} \prod_{p \mid D_2} \chi_{D_1}(p)^{\nu_p(m)} p^{\nu_p(m) s}\\
\times \prod_{p \nmid D} \chi_{D_2}(p)^{\nu_p(m)} \frac{1- (\chi_D(p) p^s )^{\nu_p(m)+1} }{ 1- \chi_D(p) p^s}\\
= \sum_{D_1D_2=D} \chi_{D_1}(m_{D_2}) \chi_{D_2}(N(\a) m_0 m_{D_1}) m_{D_2}^{s} \prod_{p \nmid D} \frac{1- (\chi_D(p) p^s )^{\nu_p(m)+1} }{ 1- \chi_D(p) p^s}.
\end{align*}
\end{proof}
With Lemma~\ref{almost-euler-sigma} for a full Euler product decomposition of $\sigma(\a,m,s)$ only a product decomposition of
\[
\sum_{D_1D_2=D} \chi_{D_1}(m_{D_2}) \chi_{D_2}(N(\a) m_0 m_{D_1}) m_{D_2}^{s}
\]
into factors corresponding to the primes $p \mid D$ is missing. The next lemma points into that direction.

\begin{lemma} \label{disc-sign-product}
Let $D=D_1D_2$ be a discriminant decomposition of $D$. Then we have
\[
\prod_{p \mid D_2} \kron{-D/p}{p}^{\nu_p(m)} \kron{ N(\a) m/m_p}{p} = \chi_{D_1}(m_{D_2}) \chi_{D_2}(N(\a) m/m_{D_2}).
\]
Here, of course $m_p := p^{\nu_p(m)}$.
\end{lemma}
\begin{proof}
Both sides of the equation are immune to changes of $m$ by a square. Hence, we can restrict the proof to squarefree $m$.
We show the equality now by induction on the number of prime divisors of $m$. For $m=1$ the equality follows from
\[
\prod_{p \mid D_2} \kron{N(\a)}{p}
= \kron{N(\a)}{|D_2|}
= \underbrace{\kron{N(\a)}{\sgn(D_2)}}_{=1} \kron{N(\a)}{D_2}
= \chi_{D_2}(N(\a)).
\]
For $m=-1$ we have on the left the additional factor
\[
\prod_{p \mid D_2} \kron{-1}{p}
= \prod_{ \substack{p \mid D_2  \\ p \equiv 3 \ (4) }   } (-1) = \sgn(D_2)
\]
which matches the factor $\chi_{D_2}(-1)$ on the right.

Now let the claim be true for squarefree $m \in \Z$. We want to prove it for $m':= m P$ with $P \nmid m$ prime. We start with the case $P \nmid D_2$. Then the product on the left changes by the factor
\[
\prod_{p \mid D_2} \kron{P}{p} = \kron{P}{|D_2|}.
\]
On the right the product changes by the factor $\chi_{D_2}(P)$. Now we show that the two factors are identical.
For $D_2>0$ we have because of $D_2 \equiv 1 \pmod 4$
\[
\kron{P}{|D_2|} = \kron{P}{D_2} = \kron{D_2}{P} = \chi_{D_2}(P).
\]
In case $D_2<0$ we consider the subcase $P \equiv 1 \pmod 4$ first:
\[
\kron{P}{|D_2|} = \kron{-D_2}{P} = \kron{-1}{P} \kron{D_2}{P} = \chi_{D_2}(P).
\]
In case $P \equiv 3 \pmod 4$ we obtain
\[
\kron{P}{|D_2|} = -\kron{|D_2|}{P} = - \kron{-1}{P} \kron{D_2}{P} = \chi_{D_2}(P).
\]
That finishes the proof of the case $P \nmid D_2$ and we proceed with the case $P \mid D_2$. Here, the product on the left side changes by the factor
\[
\kron{-D/P}{P} \prod_{ \substack{p \mid D_2 \\ p \ne P } } \kron{P}{p}
= \kron{-D/P}{P} \kron{P}{|D_2|/P}.
\]
On the right we obtain the additional factor $\chi_{D_1}(P)$.
Thus, we have to show
\[
\kron{D_1}{P} = \kron{-D/P}{P} \kron{P}{|D_2|/P}.
\]
By canceling $(D_1/P)$ from the equation we are left to show
\[
\kron{-D_2/P}{P} =  \kron{P}{|D_2|/P}.
\]
In case $P \equiv 1 \pmod 4$ we have
\[
\kron{-D_2/P}{P}
= \kron{|D_2|/P}{P}
= \kron{P}{|D_2|/P}.
\]
Now let $P \equiv 3 \pmod 4$ and $D_2<0$. Then it holds
\[
\kron{-D_2/P}{P}
= \kron{|D_2|/P}{P}
= \kron{P}{|D_2|/P}.
\]
In the last case $P \equiv 3 \pmod 4$ and $D_2>0$ we have
\[
\kron{-D_2/P}{P}
= - \kron{|D_2|/P}{P}
= \kron{P}{|D_2|/P}.
\]
\end{proof}

\begin{corollary} \label{D-factor-euler} 
We have
\begin{align*}
\prod_{p \mid D}  \br{1 + \kron{-D/p}{p}^{\nu_p(m)} \kron{N(\a) m/m_p}{p} m_p^s }
=\sum_{D_1D_2=D} \chi_{D_1}(m_{D_2}) \chi_{D_2}(N(\a) m_0 m_{D_1}) m_{D_2}^{s}.
\end{align*}
\end{corollary}
\begin{proof}
By expanding the product on the left the summands are in bijection to the discriminant decompositions $D=D_1D_2$. For each decomposition the associated summand is given by
\[
\prod_{p \mid D_2}  \kron{-D/p}{p}^{\nu_p(m)} \kron{N(\a) m/m_p}{p} m_p^s.
\]
Hence, Lemma \ref{disc-sign-product} proves the claim.
\end{proof}

\begin{corollary}
The divisor sum $\sigma(\a,m,s)$ has the Euler product expansion
\begin{align*}
|m|^{(1-s)/2} \prod_{p \mid D}  \br{1 + \kron{-D/p}{p}^{\nu_p(m)} \kron{N(\a) m/m_p}{p} m_p^s }\prod_{p \nmid D} \frac{1- (\chi_D(p) p^s )^{\nu_p(m)+1} }{ 1- \chi_D(p) p^s}.
\end{align*}
\end{corollary}
\begin{proof}
Combine Lemma~\ref{almost-euler-sigma} and Corollary~\ref{D-factor-euler}.
\end{proof}

\section{Euler product comparison} 

\begin{theorem} \label{Gb-series} 
We have for $m \ne 0$ and $\Re(s)>2$
\[
\sum_{b=1}^\infty G^b(\a,m,0)b^{-s} =|m|^{-s/2}  \frac{\zeta(s-1)}{L(s,\chi_D)} \sigma(\a,m,1-s).
\]
\end{theorem}
\begin{proof}
We write down both sides as Euler product and see that the factors coincide. We use again the abbreviation $q:=p^{1-s}$ and obtain
\[
\frac{\zeta(s-1)}{L(s,\chi_D)}
= \prod_p \frac{1-\chi_D(p)p^{-s}}{1-p^{-(s-1)}}
= \prod_p \frac{p-\chi_D(p)q}{p(1-q)}.
\]
Put together with the divisor sum we get for the right side
\begin{align*}
\prod_p \frac{p-\chi_D(p)q}{p(1-q)}\prod_{p \mid D}  \br{1 + \kron{-D/p}{p}^{\nu_p(m)} \kron{N(\a) m/m_p}{p} m_p^{1-s} }
\prod_{p \nmid D} \frac{1- (\chi_D(p) q )^{\nu_p(m)+1} }{ 1- \chi_D(p) q}.
\end{align*}
By equation~\eqref{first-euler} and Proposition \ref{rep-prime-series} the left side equals
\begin{align*}
&\prod_{p \mid D} \br{ \frac{1}{p} \sum_{r=0}^\infty N_{p^{r+1}}(\a,m) p^{-rs} } \prod_{p \nmid D} \br{ \sum_{r=0}^\infty N_{p^r}(\a,m) p^{-rs} }\\
= &\prod_{p \mid D} \br{  \frac{1 + \kron{-D/p}{p}^{\nu_p(m)} \kron{N(\a) m/m_p}{p} q^{\nu_p(m)} }{1-q} }
 \prod_{p \nmid D} \br{\frac{p-\chi_D(p)q}{p(1-q)}  \frac{1- \chi_D(p)^{\nu_p(m)+1} q^{\nu_p(m)+1}}{1- \chi_D(p) q} }.
\end{align*}
As predicted, the factors coincide.
\end{proof}
\begin{remark}
We can even interpret formula~\eqref{Gseries-mne0} for $m=0$ to obtain formula~\eqref{Gseries-m0} from Corollary~\ref{G-series-m0} by
\begin{align*}
|m|^{-s/2}  \sigma(\a,m,1-s)
&= |m|^{-s/2}  \cdot |m|^{s/2} \sum_{d \mid m} d^{1-s} \prod_{p \mid D} (\chi_{D(p)}(d) + \chi_{D(p)}(N(\a)m/d))\\
&= \sum_{d = 1}^\infty d^{1-s} \chi_D(d) = L(s-1,\chi_D).
\end{align*}
\end{remark}


\end{document}